\documentclass[ 11pt]{article}

\usepackage[utf8]{inputenc}  
\usepackage[T1]{fontenc}         
\usepackage{geometry}         
\usepackage[english]{babel}  
\usepackage{cite}
\usepackage{amssymb}   

\usepackage{amsmath} 
\usepackage{amscd}
\usepackage{graphicx, color}    
\usepackage{amsthm}                 
\usepackage{latexsym}     
\usepackage[all]{xy}

\newtheorem{thm}{Theorem}[section]
\newtheorem*{thm*}{Theorem}
\newtheorem*{thx}{Acknowledgements}

\newtheorem*{Existence}{Theorem}

\newtheorem*{cor*}{Corollary}
\newtheorem{prop}[thm]{Proposition}
\newtheorem{lem}[thm]{Lemma}

\theoremstyle{definition}
\newtheorem{definition}[thm]{Definition}

\newtheorem{rmk}[thm]{Remark}
\newtheorem{notation}[thm]{Notation}

\def\lquotient#1#2{%
\makeatletter
\lower.6ex\hbox{$#1$}\backslash\raise.3ex\hbox{$#2$}%
\makeatother
}															

\def\rquotient#1#2{%
\makeatletter
\raise.6ex\hbox{$#1$}/\raise.3ex\hbox{$#2$}%
\makeatother
}

\newcommand{\cB}{{\mathcal B}}
\newcommand{\cC}{{\mathcal C}}

\newcommand{\cX}{{\mathcal X}}
\newcommand{\cY}{{\mathcal Y}}

\newcommand{\ra}{\rightarrow}

\newcommand{\hra}{\hookrightarrow}

\title{\textbf{Combination of classifying spaces for proper actions}}           
\author{Alexandre Martin\footnote{This research is supported by the European Research Council (ERC) grant of Goulnara Arzhantseva, grant agreement n$^o$ 259527.}}
\date{}

\begin{document}
\maketitle

\begin{abstract} Given a group action on a simplicial complex such that each simplex stabiliser admits a cocompact model of classifying space for proper actions, we give conditions implying the existence of a cocompact model of classifying space for the whole group. This is used to generalise previous combination results for boundaries of groups and hyperbolic groups.\end{abstract}




\section{Introduction.}

A central problem of geometric group theory is to understand, given a group action on a simplicial complex, to what extent the properties of the group come from properties of the stabilisers of simplices. In this article, we focus on the existence of a cocompact model of classifying space for proper actions. In the case of an amalgamated product or HNN extension, Scott and Wall \cite{ScottWall} construct an explicit model of classifying space for proper actions as a \textit{tree of spaces} over the Bass--Serre tree of the splitting. We prove the following combination theorem for complexes of groups of arbitrary dimension:

\begin{thm*}
Let $G(\cY)$ be a developable complex of groups over a finite simplicial complex $Y$, such that:
\begin{itemize}
\item for every finite subgroup of the fundamental group of $G(\cY)$, the associated fixed-point set in the universal cover of $G(\cY)$ is contractible,
\item every simplex stabiliser admits a cocompact model of classifying space for proper actions.
\end{itemize}  
Then there exists a cocompact model of classifying space for proper actions for the fundamental group of $G(\cY)$, obtained as a complex of spaces over the universal cover of $G(\cY)$, and with fibres classifying spaces for proper actions of the local groups of $G(\cY)$.
\end{thm*}

Certain compactifications of cocompact models of classifying spaces for proper actions, when they exist, have proved to be very useful as they imply for instance the Novikov conjecture for the group \cite{FarrellLafontEZStructures}. In \cite{MartinBoundaries}, the author studied the asymptotic topology of a group acting cocompactly on a non-positively curved simplicial complex by means of the asymptotic topology of the stabilisers of simplices. As an example, the following theorem was proved: 

\begin{thm*}[M.\cite{MartinBoundaries}]
Let $G(\cY)$ be a non-positively curved complex of groups over a finite piecewise-Euclidean complex $Y$, such that there exists a complex of classifying spaces compatible with $G(\cY)$. Let $G$ be the fundamental group of $G(\cY)$ and $X$ be a universal covering of $G(\cY)$. Assume that:
\begin{itemize}
\item The universal covering $X$ is hyperbolic (for the associated piecewise-Euclidean structure),
\item The local groups are hyperbolic and all the local maps are quasiconvex embeddings,
\item The action of $G$ on $X$ is acylindrical. 
\end{itemize}
Then $G$ is hyperbolic and the local groups embed in $G$ as quasiconvex subgroups. \qed
\end{thm*}

To prove such a theorem, the first step is to combine the classifying spaces for proper actions of the various stabilisers of simplices (in this particular case, Rips complexes) into a classifying space for proper actions of the fundamental group of the complex of groups; the appropriate data used to construct such a classifying space is the notion of a \textit{complex of classifying spaces compatible with a complex of groups} mentioned in the previous theorem, and introduced in \cite{MartinBoundaries}. Such a classifying space is constructed as \textit{complex of spaces} over the universal cover of the complex of groups. With such a model at hand, we can then try to understand the asymptotic topology of the whole space by means of the asymptotic topology of its smaller pieces. 

Constructing a complex of classifying spaces compatible with a given complex of groups is a non-trivial problem. This can be carried out by ad hoc constructions when the combinatorics of the underlying complex of groups is quite simple (for instance, in the case of a simple complex of hyperbolic groups \cite{MartinBoundaries}, or in the case of metric small cancellation over a graph of groups \cite{MartinSmallCancellationClassifying}), but can prove to be much harder in general. There are many examples of groups acting on simplicial complexes for which even the simplicial structure of the quotient complex is hard to describe: the action of the mapping class group of a hyperbolic surface on its curve complex, the action of a group admitting a codimension one subgroup on the associated CAT(0) cube complex, etc. Thus, if ones wants to use the previous combination theorem to study groups through their non-proper actions on simplicial complexes in general, it would be preferable to have a way to construct compatible complexes of classifying spaces which does not rely too much on the combinatorics of the associated complex of groups. 

 In this article, we give a general procedure for constructing such objects.

\begin{Existence} Let $G(\cY)$ be a complex of groups over a finite simplicial complex $Y$. Then there exists a compatible complex of classifying spaces, and we can require the local maps to be embeddings. If all the local groups of $G(\cY)$ admit cocompact models of classifying spaces for proper actions, the fibres can be taken as cocompact models of classifying spaces for proper actions as well. 
\end{Existence}

In particular, each theorem of \cite{MartinBoundaries} holds  with the assumption of the existence of a compatible complex of classifying spaces removed. For instance, the previous combination theorem can be reformulated in the following way:

\begin{cor*}
Let $G$ be a group acting without inversion and cocompactly on a piecewise-Euclidean simplicial complex such that:\begin{itemize}
\item the complex $X$ is hyperbolic and CAT(0),
\item the stabilisers of simplices are hyperbolic and they embed into one another as quasiconvex subgroups,
\item the action of $G$ on $X$ is acylindrical. 
\end{itemize}
Then $G$ is hyperbolic. Furthermore, the stabilisers of simplices embed in $G$ as quasiconvex subgroups.
\end{cor*}

To obtain compatible complexes of classifying spaces, we follow a construction due to Haefliger \cite{HaefligerExtension}: given a complex of groups $G(\cY)$ over a simplicial complex $Y$ with contractible universal cover, he constructs an Eilenberg--MacLane space for the fundamental group of $G(\cY)$ as a complex of spaces over $Y$. As we want to allow groups with torsion, the point of view adopted here is slightly different. Instead of reasoning over $Y$,  we will be working over some particular subcomplexes of the universal cover of $G(\cY)$, called \textit{blocks}. The fibres will not be Eilenberg--MacLane spaces but classifying spaces for proper actions, and all the constructions will be made equivariant.

It should be noted that, although this article is written in the framework of complexes of groups over simplicial complexes for simplicity reasons, the constructions carry over without any essential change to the case of complexes of groups over \textit{polyhedral} complexes.\\

The article is organised as follows. In Section \ref{Section2}, we review a few elementary facts on complexes of groups. In Section \ref{Section4}, we define the block associated to a simplex and study the induced complex of groups. In Section \ref{Section3}, we recall the definition of a complex of spaces (in the sense of Corson \cite{CorsonComplexesofGroups}) and of a complex of classifying spaces compatible with a complex of groups \cite{MartinBoundaries}. Section \ref{Section5} is devoted to the construction of a compatible complex of classifying spaces.

\begin{thx}
I would like to thank F. Haglund and T. Januszkiewicz for pointing out this problem, as well as F. Haglund for many helpful related discussions.
\end{thx}

\section{Background on complexes of groups.}
\label{Section2}
\subsection{Definitions.}
Complexes of groups are a high-dimensional generalisation of graphs of groups, that is, objects encoding group actions on arbitrary simplicial complexes, and were introduced by Gersten--Stallings \cite{GerstenStallings}, Corson \cite{CorsonComplexesofGroups} and Haefliger \cite{HaefligerOrbihedra}. Haefliger defined a notion of complexes of groups over more general objects called \textit{small categories without loops} (abbreviated \textit{scwols}), a combinatorial generalisation of polyhedral complexes. We recall here basic definitions and properties of complexes of groups. For a deeper treatment of the material covered in this section, we refer the reader to \cite{BridsonHaefliger}.
\begin{definition}[small category without loop]
A  \textit{small category without loop} (briefly a \textit{scwol}) is a set $\cY$ which is the disjoint union of a set $V(\cY)$ called the vertex set of $\cY$, and a set $A(\cY)$ called the  set of edges\footnote{In the literature, the set of edges is usually denotes $E(\cY)$. Here however, as the letter $E$ will be used to denote classifying spaces (or spaces constructed out of such classifying spaces), we use the French notation $A(\cX)$ so as to avoid confusions.} of $\cY$, together with maps
$$i: A(\cY)  \ra V(\cY) \mbox{ and } t: A(\cY) \ra V(\cY).$$
For an edge $a \in A(\cY)$, $i(a)$ is called the initial vertex of $a$ and $t(a)$ the terminal vertex of $a$. 

For $k \geq 1$, let $A^{(k)}(\cY)$ be the set of sequences $(a_k, \ldots, a_1)$ of edges of $\cY$ such that $i(a_{i+1}) = t(a_i)$ for $1 \leq i <k$ (the sequence of edges $a_k, \ldots, a_1$ is said to be composable). For $A=(a_k, \ldots, a_1) \in A^{(k)}(\cY)$, we set $i(A):= i(a_1)$ and $t(A):=t(a_k)$. By convention, we set $A^{(0)}(\cY) = V(\cY)$.

A third map 
$$ A^{(2)}(\cY) \ra A(\cY) $$
is given that associates to a pair $(b,a)$ of composable edges an edge $ba$ called their concatenation or composition. These maps are required to satisfy the following conditions:
\begin{itemize}
 \item For every $(b,a) \in A^{(2)}(\cY)$, we have $i(ba)=i(a)$ and $t(ba)=t(b)$; 
 \item For every $(c,b,a) \in A^{(3)}(\cY)$, we have $(cb)a = c(ba)$ (and the composition is simply denoted $cba$).
 \item For every $a \in A(\cY)$, we have $t(a) \neq i(a)$.
\end{itemize}
\end{definition}

A fundamental example of scwol is the following:

\begin{definition}[simplicial scwol associated to a simplicial complex] If $Y$ is a simplicial complex, a scwol $\cY$ is naturally associated to $Y$ in the following way: 
\begin{itemize}
 \item $V(\cY)$ is the set of simplices of $Y$, 
 \item $A(\cY)$ is the set of pairs $(\sigma, \sigma') \in V(\cY)^2$ such that $\sigma \subset \sigma'$.
 \item For a pair $a=(\sigma, \sigma') \in A(\cY)$, we set $i(a) = \sigma'$ and $t(a)=\sigma$.
 \item For composable edges $b=(\sigma, \sigma')$ and $a=(\sigma', \sigma'')$, we set $ba=(\sigma, \sigma'')$.
\end{itemize}
We call $\cY$ the \textit{simplicial scwol associated to $X$}.
\end{definition}

In what follows, we will often omit the distinction between a simplex $\sigma$ of $Y$ and the associated vertex of $\cY$.

\begin{definition}[Geometric/Simplicial realisation of a scwol]
For integers $k \geq 2$ and $0 \leq i \leq k$, we define maps $\partial_i: A^{(k)}(\cY) \ra A^{(k-1)}(\cY)$ as follows: 
\begin{align*} 
\partial_0(a_k, \ldots, a_1) &= (a_{k}, \ldots, a_ 2)
 \\ \partial_i(a_k, \ldots, a_1) &= (a_{k}, \ldots,a_{i+1}a_i, \ldots, a_1 ) ~~1  \leq i < k
 \\ \partial_k(a_k, \ldots, a_1) &= (a_{k-1}, \ldots, a_1). 
\end{align*}
For $k=1$, we set $\partial_0 a = i(a)$ and $\partial_1(a) = t(a)$.

Let $\Delta^k$ be the standard Euclidean $k$-simplex, that is, the set of elements $(t_0, \ldots, t_k)$ with $ t_i \geq 0$ and $\sum_i t_i=1$. For $k \geq 1$ and $0 \leq i \leq k$, we define maps $d_i: \Delta^{k-1} \ra \Delta^k$ by sending $(t_0, \ldots, t_{k-1})$ to $(t_0, \ldots, t_{i-1},0,t_i, \ldots, t_{k-1})$. 

The \textit{geometric realisation} of the scwol $\cY$ is the space obtained from the disjoint union 
$$\underset{k \geq 0, A \in A^{(k)}(\cY)}{\coprod} \{A\} \times \Delta^k$$
by identifying pairs of the form $(\partial_iA,x)$ and $(A, d_i(x))$; this is a piecewise-Euclidean simplicial complex. We call the underlying simplicial complex the \textit{simplicial realisation} of $\cY$. 

In what follows, we will make no difference between simplicial and geometric realisations.
\end{definition}

\begin{rmk}
The simplicial realisation of the scwol associated to a simplicial complex $Y$ is naturally isomorphic to the first barycentric subdivision $Y'$ of $Y$.
\end{rmk}

\begin{definition}[Complex of groups \cite{BridsonHaefliger}]
 Let $\cY$ be a scwol. A \textit{complex of groups $G(\cY)= (G_\sigma, \psi_a, g_{b,a})$ over $\cY$} is given by the following data: 
\begin{itemize}
 \item for each vertex $\sigma$ of $\cY$, a group $G_\sigma$ called the \textit{local group} at $\sigma$,
 \item for each edge $a$ of $\cY$, an injective homomorphism $\psi_a: G_{i(a)} \ra G_{t(a)}$,
 \item for each pair of composable edges $(b,a)$ of $\cY$, a \textit{twisting element} $g_{b,a} \in G_{t(b)}$,
\end{itemize}
with the following compatibility conditions:
\begin{itemize}
 \item for every pair $(b,a)$ of composable edges of $\cY$, we have 
$$ \mbox{Ad}(g_{b,a}) \psi_{ba} =  \psi_b \psi_a,$$
where $\mbox{Ad}(g_{b,a}): g \mapsto g_{b,a} \cdot g \cdot g_{b,a}^{-1}$ is the conjugation by $g_{b,a}$ in $G_{t(b)}$; 
 \item For $(c,b,a) \in A^{(3)}(\cY)$, the following cocycle condition holds:
$$ \psi_c(g_{b,a})g_{c,ba}= g_{c,b}g_{cb,a}.$$
\end{itemize}
\end{definition}

\begin{notation}
If $a$ is an edge of $\cY$ corresponding to an inclusion $\sigma\subset \sigma'$, we will sometimes write $\psi_{\sigma, \sigma'}$ in place of $\psi_a$. By convention, we also define $\psi_{\sigma, \sigma}$ as the identity map of $G_\sigma$.
\end{notation}

\begin{definition}[Morphism of complex of groups]
 Let $Y, Y'$ be simplicial complexes, $\cY$ (resp. $\cY'$) the associated simplicial scwols, $f: Y \ra Y'$ a non-degenerate simplicial map (that is, the restriction of $f$ to any simplex is a homeomorphism on its image) , and $G(\cY)$ (resp. $G(\cY')$) a complex of groups over $Y$ (resp. $Y'$). A \textit{morphism} $F=(F_\sigma, F(a))$ : $G(\cY) \ra G(\cY')$ over $f$ consists of the following:
\begin{itemize}
 \item for each vertex $\sigma$ of $\cY$, a homomorphism $F_{\sigma}: G_\sigma \ra G_{f(\sigma)}$, 
 \item for each edge $a$ of $\cY$, an element $F(a) \in G_{t(f(a))}$ such that
   
$$ \mbox{Ad}(F(a))\psi_{f(a)}F_{i(a)} = F_{t(a)}\psi_a,$$

  \item if $(b,a)$ is a pair of composable edges of $\cY$, we have
$$ F_{t(b)}(g_{b,a})F(ba) = F(b)\psi_{f(b)}(F(a))g_{f(b),f(a)}.$$
\end{itemize}
If all the $F_\sigma$ are isomorphisms, $F$ is called a \textit{local isomorphism}. If in addition $f$ is a simplicial isomorphism, $F$ is called an \textit{isomorphism}.
\end{definition}

\begin{definition}[morphism from a complex of groups to a group]
 Let $G(\cY)$ be a complex of groups over a scwol $\cY$ and $G$ a group. A \textit{morphism $F = (F_\sigma, F(a))$ from $G(\cY)$ to $G$} consists of a homomorphism $F_\sigma: G_\sigma \ra G$ for every $\sigma \in V(\cY)$ and an element $F(a) \in G$ for each $a \in E(\cY)$ such that
\begin{itemize}
 \item for every $a \in E(\cY)$, we have $F_{t(a)} \psi_a = \mbox{Ad}(F(a)) F_{i(a)}$,
 \item for every pair $(b,a)$ of composable edges of $\cY$, we have $F_{t(b)}(g_{b,a})F(ba)=F(b)F(a)$.
\end{itemize}
\end{definition}

\subsection{Developability.}

\begin{definition}[Complex of groups associated to an action without inversion of a group on a simplicial complex \cite{BridsonHaefliger}]
 Let $G$ be a group acting without inversion by simplicial isomorphisms on a simplicial complex $X$, let $Y$ be the quotient space and $p:X \ra Y$ the natural projection. Up to a barycentric subdivision, we can assume that $p$ restricts to a embedding on every simplex, yielding a simplicial structure on $Y$. Let $\cY$ be the simplicial scwol associated to $Y$.

For each vertex $\sigma$ of $\cY$, choose a simplex $\widetilde{\sigma}$ of $X$ such that $p(\widetilde{\sigma})=\sigma$. As $G$ acts without inversion on $X$, the restriction of $p$ to any simplex of $X$ is a homeomorphism on its image. Thus, to every simplex $\sigma'$ of $Y$ contained in $\sigma$, there is a unique $\tau$ of $X$ and contained in $\widetilde{\sigma}$, such that $p(\tau) = \sigma'$ To the edge $a= (\sigma, \sigma')$ of $\cY$ we then choose an element $h_a \in G$ such that $h_a.\tau = \widetilde{\sigma'}$. A \textit{complex of groups $G(\cY)=(G_\sigma, \psi_a, g_{b,a})$ over $Y$ associated to the action of $G$ on $X$} is given by the following:
\begin{itemize}
 \item for each vertex $\sigma$ of $\cY$, let $G_\sigma$ be the stabiliser of $\widetilde{\sigma}$, 
 \item for every edge $a$ of $\cY$, the homomorphism $ \psi_a: G_{i(a)} \ra G_{t(a)}$ is defined by 
$$ \psi_a(g) = h_agh_a^{-1},$$
 \item for every pair $(b,a)$ of composable edges of $\cY$, define
$$g_{b,a} = h_bh_ah_{ba}^{-1}.$$
\end{itemize}
Moreover, there is an associated morphism $F=(F_\sigma, F(a))$ from $G(\cY)$ to $G$, where $F_\sigma: G_\sigma \ra G$ is the natural inclusion and $F(a) = h_a$.
\label{inducedcomplexofgroups}
\end{definition}

\begin{definition}[Developable complex of groups]
 A complex of groups over a simplicial complex $Y$ is \textit{developable} if it is isomorphic to the complex of groups associated to an action without inversion on a simplicial complex.
\end{definition}

There is the following algebraic characterization of developability: 

\begin{thm}[Theorem III.$\cC$.2.15 of \cite{BridsonHaefliger}]
A complex of groups $G(\cY)$ is developable if and only if there exists a morphism from $G(\cY)$ to some group which is injective on the local groups. \qed
\label{developable}
\end{thm}

\section{Induced complex of groups over a block.}
\label{Section4}
Unlike in Bass-Serre theory, not every complex of groups is developable. However, non-developability is a global phenomenon, a complex of groups being always developable around a vertex. In this section, we describe for every simplex $\sigma$ of $Y$, a sub-scwol associated to $\sigma$, called a \textit{block}, such that the induced complex of groups is developable.

\subsection{The block associated to a simplex.}

Given a simplex $\sigma$ of $Y$, we consider the consider the sub-scwol $\cB(\sigma) \subset \cY$ whose vertex set consists of those simplices of $Y$ containing $\sigma$ and whose set of edges consists of those edges of $\cY$ whose initial and terminal vertices are in $V(\cB(\sigma))$. The simplicial realisation of $\cB(\sigma)$ is a simplicial complex $B(\sigma)$, called the \textit{block} associated to $\sigma$, which is isomorphic to the subcomplex  $\{\sigma\} \star \mbox{lk}(\{\sigma\}, Y')$ of the first barycentric subdivision $Y'$ of $Y$ (where $\{\sigma\}$ denotes the vertex of $Y'$ corresponding to the simplex $\sigma$). 

\begin{figure}[ht!]
\begin{center}
\scalebox{0.6}{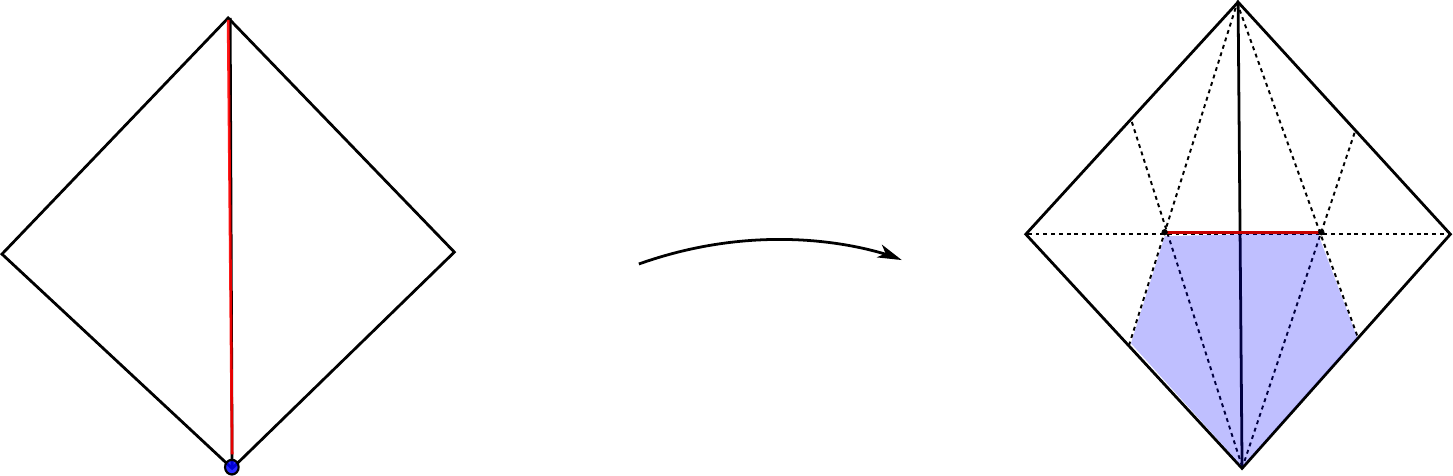}
\caption{On the left, a simplicial complex with a chosen vertex (blue) and edge (red). On the right, the associated blocks.}
\end{center}
\end{figure}

For every inclusion of simplices $\sigma \subset \sigma'$, we have an inclusion of blocks $B(\sigma') \subset B(\sigma)$. Moreover, since the block $B(\sigma)$ is simplicially a cone over the link $\mbox{lk}(\{\sigma\}, Y')$, it is contractible.

\subsection{The local development.}

We denote by $G\big(\cB(\sigma)\big)$ the induced complex of groups over $\cB(\sigma)$, that is, the pullback of $G(\cY)$ under the inclusion $\cB(\sigma) \hra \cY$.

We define a morphism $F_\sigma$ from $G\big(\cB(\sigma)\big)$ to $G_\sigma$ as follows. For every vertex $\tau$ of $\cB(\sigma)$ (that is, for every simplex $\tau$ of $Y$ containing $\sigma$), the map $\big(F_\sigma\big)_\tau$ is the map $\psi_{\sigma, \tau}$. Let $a$ be an edge of $\cB(\sigma)$. If $t(a)= \sigma$, we define $F_\sigma(a)$ as the identity element of $G_\sigma$. Otherwise, let $b$ be the edge from $t(a)$ to $\sigma$, and we set $F_\sigma(a) = g_{b, a}.$

This defines a morphism which is injective on the local groups, so that $G\big(\cB(\sigma)\big)$ is developable by Theorem \ref{developable}. We denote by $\widetilde{\cB(\sigma)}$ the development of $\cB(\sigma)$ associated to this morphism, and by $\widetilde{B(\sigma)}$ the simplicial realisation of $\widetilde{\cB(\sigma)}$. We dispose of the following description of $\widetilde{\cB(\sigma)}$ (see Theorem III.$\cC$.2.13 of \cite{BridsonHaefliger}):
$$V(\widetilde{\cB(\sigma)}) = \underset{\tau \in V(\cB(\sigma))}{\coprod} \bigg( \lquotient{\psi_{\sigma, \tau}(G_\tau)}{G_\sigma} \times \{\tau\} \bigg),$$
$$A(\widetilde{\cB(\sigma)}) = \underset{a \in A(\cB(\sigma))}{\coprod} \bigg( \lquotient{\psi_{\sigma, i(a)}(G_{i(a)})}{G_\sigma} \times \{a\} \bigg);$$
the initial and terminal vertices are defined as follows:
$$ i\big([g],a\big)= \big([g], i(a)\big),$$
$$ t\big([g],a\big)= \big([g]F_\sigma(a)^{-1}, t(a)\big);$$
the composition is
$$\big([g],b\big)\big([gF_\sigma(a)^{-1}], a\big) = \big([g], ba\big) $$
where $(b,a)$ is a pair of composable edges of $\cB(\sigma)$ and $g$ an element of $G_\sigma$. We also have, for an integer $k \geq 1$:
$$A^{(k)}(\widetilde{\cB(\sigma)}) = \underset{A \in A^{(k)}(\cB(\sigma))}{\coprod} \bigg( \lquotient{\psi_{\sigma, i(A)}(G_{i(A)})}{G_\sigma} \times \{A\} \bigg).$$
The simplicial realisation $\widetilde{B(\sigma)}$ has the structure of a simplicial cone, hence is contractible. If $G(\cY)$ is developable, such a cone is simplicially isomorphic to the block associated to any lift $\widetilde{\sigma}$ of $\sigma$.

\section{Complexes of spaces}
\label{Section3}
Complexes of spaces are a high-dimensional generalisation of graphs of spaces. They were considered, in relation with complexes of groups, by Corson \cite{CorsonComplexesofGroups} and Haefliger \cite{HaefligerExtension}. Here, we will only be dealing with Corson's definition as it is more flexible (see Remark \ref{HaefligervsCorson} for further details). 

\begin{definition}[Complexes of spaces in the sense of Corson \cite{CorsonComplexesofGroups}]
Let $Y$ be a connected simplicial complex. A complex of spaces over $Y$ is a connected simplicial complex $X$ together with a simplicial map $p: X \ra Y$ such that for each open simplex $\sigma$ of $X$, $p^{-1}(\sigma)$ is a connected subcomplex of $X$ of the form $X_\sigma \times \sigma$, where $X_\sigma$ is the pre-image of the centre of the simplex $\sigma$ (the \textit{fibre} of $\sigma$), and such that for every subface $\tau$ of $\sigma$, the induced map on fundamental groups $\pi_1(X_\sigma) \ra \pi_1(X_\tau)$, obtained by translating the base point along a path in $X_\sigma$, is injective. Furthermore, we require that the topology on $X$ be coherent with subsets of the form $p^{-1}(\overline{\sigma})$, where $\overline{\sigma}$ is a closed simplex of $Y$.
\end{definition}

\begin{rmk}
Haefliger's definition imposes in addition that the restriction of the projection $p:X \ra Y$ to the $1$-skeleton of $X$ admits a section $s: Y^{(1)} \ra X^{(1)}$. This however turns out to be incompatible with the equivariance wanted in our constructions.
\label{HaefligervsCorson}
\end{rmk}

Let $G(\cY)$ be a developable complex of groups over a simplicial complex $Y$ with a contractible universal cover. Haefliger \cite{HaefligerExtension} and independently Corson construct an Eilenberg--MacLane space for the fundamental group of $G(\cY)$ as a complex of spaces over $Y$, with fibres Eilenberg--MacLane spaces for the local groups. As we are interested in cocompact models of classifying spaces for proper action, we adopt a slightly different point of view. 

\begin{definition}[Complex of classifying spaces compatible with a complex of groups \cite{MartinBoundaries}]
Let $G(\cY) = (G_\sigma, \psi_a, g_{b,a})$ be a complex of groups over a simplicial complex $Y$. A \textit{complex of classifying spaces $\underbar{E}G(\cY)$ compatible with the complex of groups $G(\cY)$} consists of the following:
\begin{itemize}
 \item For every vertex $\sigma$ of $\cY$, a space $\underbar{E}G_\sigma$ (called a \textit{fibre}) that is a model of classifying space for proper actions of the local group $G_\sigma$,
 \item For every edge $a$ of $\cY$, a $\psi_a$-equivariant map $\phi_a: \underbar{E}G_{i(a)} \ra \underbar{E}G_{t(a)}$, that is, for every $g \in G_{i(a)}$ and every $x \in \underbar{E}G_{i(a)}$, we have 
$$ \phi_{a}(g.x) = \psi_{a}(g).\phi_{a}(x),$$
and such that for every pair $(b,a)$ of composable edges of $\cY$, we have:
$$g_{b,a} \cdot \phi_{ba} = \phi_b  \phi_a.$$
\end{itemize}
\label{EZcomplexofspaces}
\end{definition}

\begin{notation}
If $a$ is an edge of $\cY$ corresponding to an inclusion $\sigma\subset \sigma'$, we will sometimes write $\phi_{\sigma, \sigma'}$ in place of $\phi_a$. By convention, we also define $\phi_{\sigma, \sigma}$ as the identity map of $\underbar{E}G_\sigma$.
\end{notation}

We emphasise that a complex of classifying spaces compatible with the complex of groups $G(\cY)$ is \textit{not} a complex of spaces over $Y$ if the twist coefficients $g_{b,a}$ are not trivial. However, the following holds: 

\begin{thm}[M. \cite{MartinBoundaries}]
Let $G(\cY)$ be a developable complex of groups over a finite simplicial complex $Y$, such that for every finite subgroup of the fundamental group of $G(\cY)$, the associated fixed-point set in the universal cover of $G(\cY)$ is contractible. Assume that there is a complex of classifying spaces $\underbar{E}G(\cY)$ compatible with $G(\cY)$. Then there exists a model of classifying space for proper actions for the fundamental group of $G(\cY)$, obtained as a complex of spaces (in the sense of Corson) over the universal cover of $G(\cY)$, and with fibres classifying spaces for proper actions of the local groups of $G(\cY)$. Furthermore, if all the fibres of $\underbar{E}G(\cY)$ are cocompact models, then the resulting complex of spaces is a cocompact model of classifying spaces for proper actions. \qed
\end{thm}

\section{The topological construction.}
\label{Section5}
From now on, we assume that we are given a complex of groups $G(\cY)$ over a finite simplicial complex $Y$, such that each local group admits a cocompact model of classifying space for proper actions. For each simplex $\sigma$ of $Y$, we choose a model of classifying space for proper actions to be a \textit{based} $G_\sigma$-space with a chosen basepoint, that is, with a chosen $G_\sigma$-orbit and a prefered point in that orbit.  For every edge $a \in A(\cY)$, we choose a based $\psi_a$-equivariant continuous map $\varphi_a$ from $\underbar{E}G_{i(a)}$ to $\underbar{E}G_{t(a)}$. Without loss of generality, we can assume that the maps $\psi_a$ preserve basepoints.

Let $k \geq 1$ be an integer, and denote by $I^k$ the $k$-dimensional cube $[0,1]^k$.

In this section, we construct by induction CW-complexes $E_0(\sigma), E_1(\sigma), \ldots$ such that for every integer $k \geq 0$, $E_k(\sigma)$ is a complex of spaces (in the sense of Corson) over the $k$-skeleton of the block $\widetilde{B(\sigma)}$. In order to construct these spaces, we need the following observation: 

\begin{lem}
Let $\sigma \subset \sigma'$ be simplices of $Y$, $k \geq 1$ an integer, and $f: \underbar{E}G_{\sigma'} \times \partial I^k \ra \underbar{E}G_\sigma$ a $\psi_{\sigma, \sigma'}$-equivariant map, where $G_{\sigma'}$ acts trivially on $I^k$. Then $f$ extends to a $\psi_{\sigma, \sigma'}$-equivariant map $F: \underbar{E}G_{\sigma'} \times I^k \ra \underbar{E}G_\sigma$. 
\label{lemmetechnique}
\end{lem}

\begin{proof}
First notice that the space $\underbar{E}G_{\sigma'} \times \partial I^k$ is a $G_{\sigma'}$-space whose isotropy groups are finite by definition of $\underbar{E}G_{\sigma'}$. Furthermore, for each such finite subgroup $F$ of $G_{\sigma'}$, the fixed-point set $(\underbar{E}G_{\sigma})^{\psi_{\sigma, \sigma'}(F)}$  is contractible by definition of $\underbar{E}G_{\sigma}$. 

We start with the case $k \geq 2$. The function $f$ is defined on  $\underbar{E}G_{\sigma'} \times \partial I^k $, so in particular on the $1$-skeleton of $\underbar{E}G_{\sigma'} \times I^k$. Using the above remarks, it is then a standard consequence of equivariant obstruction theory that the map $f: \underbar{E}G_{\sigma'} \times \partial I^k \ra \underbar{E}G_{\sigma}$ equivariantly extends to $\underbar{E}G_{\sigma'} \times I^{k}$.

For $k=1$, we are given two equivariant maps $f_1, f_2: \underbar{E}G_{\sigma'} \ra \underbar{E}G_{\sigma}$, which coincide on the chosen $G_{\sigma'}$-orbit by assumption. Thus we can naturally extend the map $f$ to the $1$-skeleton of $\underbar{E}G_{\sigma'} \times I$, and one then concludes with the same reasoning as above.
\end{proof}

Let $A=(a_k, \ldots, a_1)$ be an element of $ A^{(k)}(\sigma)$. Following Haefliger \cite{HaefligerExtension}, we denote by $r_k: I^k \ra \Delta^k$ the simplicial map which sends the vertex $(0, \ldots, 0) \in I^k$ on the vertex $(0, \ldots, 0,1) \in \Delta^k$, and which sends the vertex $(t_1, \ldots, t_{i-1}, 1, 0, \ldots, 0)$ on the vertex $(0, \ldots, 0,1,0, \ldots, 0)$ (where the single $1$ is at position $i+1$ starting from the right). Note that $r_k$ realises a homeomorphism between the interior of $I^k$ and the interior of $\Delta^k$.\\

We define the space 
\begin{center}
$E_A :=$ \raise.5ex\hbox{$\bigg( G_\sigma \times \underbar{E}G_{i(A)} \bigg)$}\big/\lower.5ex\hbox{$\sim$} $\times \{A\} \times I^k,$
\end{center}
where $ (g,g'x) \sim (g \psi_a(g'), x)$ for every $g \in G_\sigma$, $x \in \underbar{E}G_{i(A)}$, $g' \in G_{i(A)}$, and where $a$ denotes the concatenation $a_k \ldots a_1$. 

\begin{rmk}
We can define an equivalence relation $\sim'$ on $ G_\sigma \times \underbar{E}G_{i(A)}\times \{A\}$ (resp. $\sim''$ on $ G_\sigma \times \underbar{E}G_{i(A)}\times \{A\} \times I^k$), yielding a space $E_A'$ (resp. $E_A''$), and such that the identity of  $ G_\sigma \times \underbar{E}G_{i(A)}\times \{A\} \times I^k$ yields a homeomorphism $E_A \ra E_A'$ (resp. $E_A \ra E_A''$). Therefore, we will sometimes write $\big([g,x,A],(t_i)_i\big)$  or $[g,x,A,(t_i)_i]$ when speaking of an element of $E_A$.
\end{rmk}

Note that there is a $G_\sigma$-equivariant map $p_A$ from $E_A$ to the $G_\sigma$-orbit of the simplex $|A|$ in $\widetilde{B(\sigma)}$, defined by 
$$p_A\big( [g,x,A,(t_i)_i] \big) = [g,A,r_k\big((t_i)_i\big)] .$$
Moreover, the preimage $p_A^{-1}(g \overset{\circ}{|A|})$ of a translate of the interior of $|A|$ is homeomorphic to the product $\underbar{E}G_{i(A)} \times \mathring{I}^k$. \\

\textbf{Step 0:} We set 
$$E_0(\sigma) = \underset{v \in V(\cB(\sigma))}{\coprod} E_v,$$
 which comes with the obvious projection to 
$$\widetilde{B(\sigma)}^{(0)}= \underset{v \in V(\cB(\sigma))}{\coprod} \{v\}.$$

\textbf{Step 1:} For an element $A\in A^{(1)}(\sigma)$ (that is, an edge $a $ of  $\cB(\sigma)$), let $\Phi_A: \partial E_A \ra E_0(\sigma)$ be the map that sends the element $[g,x,a,0]$ to $[g,x,i(a)]$ and $[g,x,a,1]$ to $[gF_\sigma(a)^{-1},\varphi_a(x),t(a)]$.
We can thus define the space $E_1(\sigma)$ as the quotient space 
\begin{center}
$E_1(\sigma) :=$ \raise.5ex\hbox{$\bigg( E_0(\sigma) \sqcup \underset{A \in A^{(1)}(\sigma)}{\coprod} E_A \bigg)$}\big/\lower.5ex\hbox{$(\Phi_A)_{A \in A^{(1)}(\sigma)}$}.
\end{center}
We check that the various maps $p_A$ can be assembled into a map $p_1: E_1(\sigma) \ra \widetilde{B(\sigma)}^{(1)}$ that makes $E_1(\sigma)$ a complex of spaces (in the sense of Corson) over $\widetilde{B(\sigma)}^{(1)}$.\\

\textbf{Step 2:} We now turn to the construction of $E_2(\sigma)$.  Using the same idea, we first want to define a map $\Phi_A: \partial E_A \ra E_1(\sigma)$ for every $A \in A^{(2)}(\sigma)$. Let $A=(a_2,a_1)$ be such a pair of composable edges. Here, there are a priori two different ways to map $\underbar{E}G_{i(a_1)}$ to $\underbar{E}G_{t(a_2)}$ in a $\psi_{a_2a_1}$-equivariant way, namely  $\varphi_{a_2a_1}$ and $ g_{a_2,a_1}^{-1} \varphi_{a_2} \varphi_{a_1}$. In view of Lemma \ref{lemmetechnique}, these maps are equivariantly homotopic.

\begin{definition}
We denote by $H_{a_2, a_1}$ an equivariant map $\underbar{E}G_{i(a_1)} \times [0,1] \ra \underbar{E}G_{t(a_2)}$ such that $H_{a_2, a_1}(\bullet,0) = \varphi_{a_2a_1}$ and $H_{a_2,a_1}(\bullet,1) = g_{a_2,a_1}^{-1} \varphi_{a_2}  \varphi_{a_1}$. 
\end{definition}

For an element $A=(a_2, a_1) \in A^{(2)}(\sigma)$, we define a map $\Phi_A: \partial E_A \ra E_1(\sigma)$ as follows: 
\begin{align*}
\Phi_A\big( [g,x,A,(t_1,t_2)] \big) &= [g,x,a_1,t_1] \mbox{~if~}  t_2=0, 
\\ \Phi_A\big( [g,x,A,(t_1,t_2)] \big) &= [g,x,a_2a_1,t_2] \mbox{~if~}  t_1=0,
\\ \Phi_A\big( [g,x,A,(t_1,t_2)] \big) &= [gF_\sigma(a_1)^{-1},\varphi_{a_1}(x),a_2,t_2] \mbox{~if~}  t_1=1,
\\ \Phi_A\big( [g,x,A,(t_1,t_2)] \big) &= [gF_\sigma(a_2a_1)^{-1},H_{a_2,a_1}(x,t_1),t(a_2)] \mbox{~if~}  t_2=1.
\end{align*}
We need to check that these definitions are compatible. The only non-trivial case to consider is when $t_1=t_2=1$, for which we get
$$[gF_\sigma(a_1)^{-1},\varphi_{a_1}(x),a_2,1]= [gF_\sigma(a_1)^{-1}F_\sigma(a_2)^{-1},\varphi_{a_2}\varphi_{a_1}(x),t(a_2)] $$
and
\begin{align*}
[gF_\sigma(a_2a_1)^{-1},H_{a_2,a_1}(x,1),t(a_2)] &= [gF_\sigma(a_2a_1)^{-1},g_{a_2,a_1}^{-1} \varphi_{a_2}  \varphi_{a_1}(x),t(a_2)] \\ &=[gF_\sigma(a_2a_1)^{-1}\psi_{a_3a_2a_1}(g_{a_2,a_1})^{-1}, \varphi_{a_2} \varphi_{a_1}(x),t(a_2)],
\end{align*}
 where $a_3$ stands for the edge (possibly empty) corresponding to the inclusion $\sigma \subset t(a_2)$. But the cocycle condition yields 
$$F_\sigma(a_2)F_\sigma(a_1) = \psi_{a_3a_2a_1}(g_{a_2,a_1})F_\sigma(a_2a_1),$$
 thus 
$$F_\sigma(a_2a_1)^{-1} \psi_{a_3a_2a_1}(g_{a_2,a_1})^{-1}= F_\sigma(a_1)^{-1}F_\sigma(a_2)^{-1}, $$
hence the equality.

Note, as it will be important for the following steps, that the following holds: 
$$\mbox{~for every~} A=(a_2, a_1) \in A^{(2)}(\sigma), \mbox{~we have~} \mbox{Im} (\Phi_A)_{|_{t_2=1}} \subset E_{t(A)}. $$

We now define $E_2(\sigma)$ as the quotient space 

\begin{center}
$E_2(\sigma) :=$ \raise.5ex\hbox{$\bigg( E_1(\sigma) \sqcup \underset{A \in A^{(2)}(\sigma)}{\coprod} E_A \bigg)$}\big/\lower.5ex\hbox{$(\Phi_A)_{A \in A^{(2)}(\sigma)}$}.
\end{center}
Here again, it is straightforward to check that the various maps $p_A$ can be assembled into a map $p_2: E_2(\sigma) \ra \widetilde{B(\sigma)}^{(2)}$ that makes $E_2(\sigma)$ a complex of spaces (in the sense of Corson) over $\widetilde{B(\sigma)}^{(2)}$.\\

\textbf{Step $k \geq 3$:} Suppose by induction that we have defined the spaces $E_0(\sigma), \ldots, E_{k-1}(\sigma)$ and the maps $\Phi_A$ for every sequence of composable edges $(a_i, \ldots, a_1), 1 \leq i<k$, satisfying the following additional condition: 
\begin{equation}
\mbox{~for every~} A \in A^{(i)}(\sigma), 1\leq i<k, \mbox{~we have~} \mbox{Im} (\Phi_A)_{|_{t_i=1}} \subset E_{t(A)} \tag{$\dagger$}
\end{equation}
Let $A=(a_k,\ldots,a_1)$ be a sequence of composable edges of $\cB(\sigma)$. Let $\partial ' I^k$ be the closure of the boundary $\partial I^k$ with the face $\{t_k=1\}$ removed, and 
\begin{center}
$\partial ' E_A :=$ \raise.5ex\hbox{$\bigg( G_\sigma \times \underbar{E}G_{i(A)} \bigg)$}\big/\lower.5ex\hbox{$\sim$} $\times \{A\} \times \partial ' I^k$ 
\end{center}
the associated subset of $\partial  E_A$. We first define a map $\Phi_A': \partial ' E_A \ra E_{k-1}(\sigma)$ as follows. The element $\Phi_A\big( [g,x,A,(t_1, \ldots, t_k)]\big)$ is defined as:
\begin{align*}
&\Phi_{a_k, \ldots, a_{i+1}a_i, \ldots, a_1} \big(  [g,x, (a_k, \ldots, a_{i+1}a_i, \ldots, a_1), (t_1,\ldots, \hat{t_i}, \ldots, t_k)]\big) &&\mbox{~for~} t_i=0 \mbox{~and~} 1 \leq i < k,
\\ &\Phi_{a_{k-1}, \ldots, a_1} \big( [g,x, (a_{k-1}, \ldots, a_1), (t_1, \ldots, t_{k-1})] \big ) && \mbox{~for~} t_k=0,
\end{align*}
and
\begin{align*}
\Phi_{a_k, \ldots, a_{i+1}} \bigg(  \Phi_{a_k, \ldots, a_{i+1}}\big( [g,x,(a_k, \ldots, a_{i+1}), (t_1, \ldots, t_{k-1},1)]\big) , (t_{i+1}, \ldots, t_k)     \bigg) 
\end{align*}
$ \mbox{for~} t_i=1 \mbox{~and~} 1 \leq i < k.$
We easily check that these definitions are compatible. Because of the condition $(\dagger)$, the restriction of $\Phi_A'$ to the set of elements of $\partial ' E_A$ with $t_k=1$  defines a $\psi_{a_k \ldots a_1}$-equivariant map 
\begin{center}
$ $ \raise.5ex\hbox{$\bigg( G_\sigma \times \underbar{E}G_{i(A)} \bigg)$}\big/\lower.5ex\hbox{$\sim$} $\times \{A\} \times \partial \{t_k=1\} \ra E_{t(A)}$ 
\end{center}
that can be extended to a $\psi_{a_k \ldots a_1}$-equivariant map 
\begin{center}
$ $ \raise.5ex\hbox{$\bigg( G_\sigma \times \underbar{E}G_{i(A)} \bigg)$}\big/\lower.5ex\hbox{$\sim$} $\times \{A\} \times  \{t_k=1\} \ra E_{t(A)} $ 
\end{center}
in view of Lemma \ref{lemmetechnique}.

We thus obtain a map $\Phi_A: \partial E_A \ra E_{k-1}(\sigma)$,  such that   $\mbox{Im} \Phi_{|_{t_k=1}} \subset E_{t(A)}$. As usual, we use these maps to define the space $E_k(\sigma)$ as the quotient space 

\begin{center}
$E_k(\sigma) :=$ \raise.5ex\hbox{$\bigg( E_{k-1}(\sigma) \sqcup \underset{A \in A^{(k)}(\sigma)}{\coprod} E_A \bigg)$}\big/\lower.5ex\hbox{$(\Phi_A)_{A \in A^{(k)}(\sigma)}$}.
\end{center}

Here again, we check that the projections $p_A$ can be combined into a map $p_k: E_k(\sigma) \ra \widetilde{B(\sigma)}^{(k)}$ that turns $E_k(\sigma)$ into a complex of spaces (in the sense of Corson) over $\widetilde{B(\sigma)}^{(k)}$, which concludes the induction.\\

Since the complex $\widetilde{B(\sigma)}$ is of finite dimension, this procedure eventually stops, and we denote by $E(\sigma)$ the final space obtained. This space is a complex of spaces (in the sense of Corson) over the block $\widetilde{B(\sigma)}$. 

\begin{prop}
For each simplex $\sigma$ of $Y$, the space $E(\sigma)$ is a model of classifying space for proper actions of $G_\sigma$. If every fibre is a cocompact model of classifying space for proper actions, then $E(\sigma)$ is a cocompact model.
\end{prop}

\begin{proof}
The space $E(\sigma)$ is a complex of spaces (in the sense of Corson) over the contractible block $\widetilde{B(\sigma)}$. Moreover, each fibre is contractible, being a classifying space for proper actions of a local group of $G\big(\cB(\sigma)\big)$. It thus follows from Proposition 3.1 of \cite{CorsonComplexesofGroups} that $E(\sigma)$ is contractible. 

For every simplex $\tau$ of the block $B(\sigma)$, the action of $G_\tau$ on $\underbar{E}G_\tau$ is properly discontinuous, so it is straightforward to check that the same holds for the action of $G_\sigma$ on $E(\sigma)$. If each of these actions is also cocompact, it follows immediately that the action is cocompact.

Finally, let $H$ be a subgroup of $G_\sigma$. Since the projection $p:E(\sigma) \ra \widetilde{B(\sigma)}$ is equivariant, the fixed-point set $E(\sigma)^H$ is a complex of spaces (in the sense of Corson) over the fixed-point set $\widetilde{B(\sigma)}^H$. The latter subcomplex is contractible since it is simplicially a cone. Moreover, the fibre over a simplex $\tau$ is equivariantly homeomorphic to the fixed-point set $\underbar{E}G_\tau^H$. It thus follows that $E(\sigma)^H$ is non-empty if and only if $H$ is finite, in which case it is contractible, which concludes the proof.
\end{proof}

We now dispose of a cocompact model of classifying space $E(\sigma)$ for every simplex $\sigma$ of $Y$. By construction of these spaces, the embeddings of blocks $B(\sigma') \hra B(\sigma)$ (for simplices $\sigma \subset \sigma'$ of $Y$) are covered by equivariant embeddings $E(\sigma')\hra E(\sigma)$, where an element of the form $[g,x, A, (t_i)_i]$ of $E(\sigma')$ is sent to the element $[\psi_{\sigma, \sigma'}(g),x, A, (t_i)_i]$ of $E(\sigma)$. We now twist these maps in order to get a complex of classifying spaces compatible with $G(\cY)$. 

\begin{thm}
There exists a complex of classifying spaces compatible with $G(\cY)$, with embeddings as local maps. Furthermore, if each local group admits a cocompact model of classifying space for proper actions, then each fibre of the associated compatible complex of classifying spaces can be chosen to be a cocompact model.
\end{thm} 

\begin{proof}
For every simplex $\sigma$ of $Y$, we define the fibre of $\sigma$ to be the space $E(\sigma)$. For every inclusion $\sigma \subset \sigma'$ of simplices, we define a $\psi_{b}$-equivariant embedding $\phi_{b}: E(\sigma') \hra E(\sigma)$ as follows, where $b$ stands for the edge of $\cB(\sigma')$ corresponding to the inclusion  $\sigma \subset \sigma'$. Let $A=(a_k, \ldots, a_1)$ be an element of $A^{(k)}(\cB(\sigma'))$ and define the edge $a$ as the concatenation $a_k\ldots a_1$. Notice that $(b,a)$ defines a pair of composable edges. We then define the restriction of $\phi_{b}$ to the subset $E_A \subset E(\sigma')$ by setting 
$$\phi_{b}\big( [g,x,A,(t_i)_i] \big) = [\psi_{b}(g)g_{b,a},x,A,(t_i)_i]. $$
One checks that these maps are compatible. 

Now let $\sigma \subset \sigma' \subset \sigma''$ be an inclusion of simplices of $Y$. Let $A=(a_k, \ldots, a_1)$ be an element of $A^{(k)}(\cB(\sigma''))$ and define the edge $a$ as the concatenation $a_k \ldots a_1$. Let $b$ be the edge of $\cB(\sigma'')$ corresponding to the inclusion $\sigma' \subset \sigma''$ and $c$ the edge of $\cB(\sigma')$ corresponding to $\sigma \subset \sigma'$. The map $\phi_{cb}$ sends an element $[g,x,A,(t_i)_i]$ to $[\psi_{cb}(g)g_{cb,a},x,A,(t_i)_i]$, while the map $\phi_{c}\phi_{b}$ sends $[g,x,A,(t_i)_i]$ to 
$$[\psi_{c}\psi_{b}(g)\psi_{c}(g_{b,a})g_{c,ba},x,A,(t_i)_i]= [g_{c,b}\psi_{cb}(g)g_{c,b}^{-1}\psi_{c}(g_{b,a})g_{c,ba},x,A,(t_i)_i].$$
Now the cocycle condition 
$$\psi_c(g_{b,a})g_{c,ba}= g_{c,b}g_{cb,a}$$
implies that 
$$\phi_{c}\phi_{b}= g_{c,b}\phi_{cb},$$
therefore turning $\big(E(\sigma), \phi_a\big)$ into a complex of classifying spaces compatible with  $G(\cY)$.
\end{proof}

\bibliographystyle{plain}
\bibliography{ComplexesOfSpaces}

\end{document}